\documentclass[a4paper,12pt,oneside]{amsart}
\usepackage{amssymb,amscd}
\usepackage[bibencoding=ascii, backend=bibtex, doi=false, url=false, isbn=false, maxbibnames=10]{biblatex}
\bibliography{references}

\usepackage{eucal}
\usepackage{graphicx}
\usepackage[utf8]{inputenc}

\usepackage{bbm}
\usepackage{dsfont}

\newcommand{\mm}{\mathfrak m}

\newcommand{\E}{\mathds{E}}
\newcommand{\V}{\mathds{V}}
\newcommand{\N}{\mathds{N}}
\newcommand{\R}{\mathds{R}}

\newcommand{\M}{\mathds{M}}
\def\P{{\mathds P}}
\newcommand{\ind}{{\mathds 1}}

\newcommand{\Nc}{\mathcal{N}}

\newcommand{\xb}{{\bf x}}

\newcommand{\yb}{{\bf y}}
\newcommand{\zb}{{\bf z}}

\newcommand{\Nb}{{\bf N}}

\DeclareMathOperator{\diam}{diam}

\newtheoremstyle{thm}{}{}%
     {\em}
     {}
     {\bf}
     {.}
     {0.5em}
     {\thmname{#1}\thmnumber{ #2}\thmnote{ #3}}

\newtheoremstyle{def}{}{}%
     {\rm}
     {}
     {\bf}
     {.}
     {0.5em}
     {\thmname{#1}\thmnumber{ #2}\thmnote{ #3}}

\theoremstyle{thm}

\newtheorem{thm}{Theorem}[section]
\newtheorem{lem}[thm]{Lemma}
\newtheorem{cor}[thm]{Corollary}
\newtheorem{prop}[thm]{Proposition}

\theoremstyle{def}

\newtheorem*{rem*}{Remark}

\let\epsilon=\varepsilon
\let\phi=\varphi
\begin{document}

\title[Concentration for Subgraph Counts]{Concentration for Poisson U-Statistics:\\ Subgraph Counts in Random Geometric Graphs}

\author{Sascha Bachmann}
\address{Sascha Bachmann \newline Institut f\"ur Mathematik \newline Universit\"at Osnabr\"uck \newline 49069 Osnabr\"uck, Germany \newline Email: sascha.bachmann@uni-osnabrueck.de}

\author{Matthias Reitzner}
\address{Matthias Reitzner \newline Institut f\"ur Mathematik \newline Universit\"at Osnabr\"uck \newline 49069 Osnabr\"uck, Germany \newline Email: matthias.reitzner@uni-osnabrueck.de}

\thanks{Both authors are partially supported by the German Research Foundation DFG-GRK 1916.}

\begin{abstract}
Concentration inequalities for subgraph counts in random geometric graphs built over Poisson point processes are proved. The estimates give upper bounds for the probabilities $\P(N\geq M +r)$ and $\P(N\leq M - r)$ where $M$ is either a median or the expectation of a subgraph count $N$. The bounds for the lower tail have a fast Gaussian decay and the bounds for the upper tail satisfy an optimality condition. A special feature of the presented inequalities is that the underlying Poisson process does not need to have finite intensity measure. 

The tail estimates for subgraph counts follow from concentration inequalities for more general local Poisson U-statistics. These bounds are proved using recent general concentration results for Poisson U-statistics and techniques based on the convex distance for Poisson point processes.
\smallskip
\end{abstract}

\keywords{Random Graphs, Subgraph Counts, Concentration Inequalities, Stochastic Geometry, Poisson Point Process, Convex Distance.}

\smallskip


\subjclass[2010]{Primary 60D05; Secondary 05C80, 60C05}
\maketitle

\section{Introduction}
Random geometric graphs have been a very active area of research for some decades. The most basic model of these graphs is obtained by choosing a random set of vertices in $\R^d$ and connecting any two vertices by an edge whenever their distance does not exceed some fixed parameter $\rho>0$. In the seminal work \cite{G_1961} by Gilbert, these graphs were suggested as a model for random communication networks. In consequence of the increasing relevance of real-world networks like social networks or wireless networks, variations of Gilbert's model have gained considerable attention in recent years, see e.g. \cite{LP_2013_1, LP_2013_2, CYG_2014, MP_2010, DF_2011}. For a detailed historical overview on the topic, we refer the reader to Penrose's fundamental monograph \cite{P_2003} on random geometric graphs.
\smallskip

For a given random geometric graph $\mathfrak{G}$ and a fixed connected graph $H$ on $k$ vertices, the corresponding \emph{subgraph count} $N^H$ is the random variable that counts the number of occurrences of $H$ as a subgraph of $\mathfrak{G}$. Note that only non-induced subgraphs are considered throughout the present work. The resulting class of random variables has been studied by many authors, see \cite[Chapter 3]{P_2003} for a historical overview on related results.
\smallskip

The purpose of the present paper is to establish \emph{concentration inequalities} for $N^H$, i.e. upper bounds on the probabilities $\P(N^H\geq M+r)$ and $\P(N_H\leq M-r)$, when the vertices of $\mathfrak{G}$ are given by the points of a Poisson point process. Our concentration results for subgraph counts, gathered in the following theorem, provide estimates for the cases where $M$ is either the expectation or a median of $N^H$.

\begin{thm} \label{main}
Let $\eta$ be a Poisson point process in $\R^d$ with non-atomic intensity measure $\mu$. Let $H$ be a connected graph on $k\geq 2$ vertices. Consider the corresponding subgraph count $N^H=N$ in a random geometric graph associated to $\eta$. Assume that $\mu$ is such that almost surely $N < \infty$. Then all moments of $N$ exist and for all $r\geq 0$,
\begin{align*}
\P(N\geq \E N + r) &\leq \exp\left(-\frac{((\E N + r)^{1/(2k)}-(\E N)^{1/(2k)})^2}{2k^2c_d}\right),\\
\P(N\leq \E N - r) &\leq \exp\left(-\frac{r^2}{2k\V N}\right),\\
\P(N> \M N+r) &\leq 2 \exp\left(-\frac{r^2}{4k^2c_d(r+\M N)^{2-1/k}}\right),\\
\P(N< \M N-r) &\leq 2 \exp\left(-\frac{r^2}{4k^2c_d(\M N)^{2-1/k}}\right),
\end{align*}
where $c_d>0$ is a constant that depends on $H$ and $d$ only, $\E N$ and $\V N$ denote the expectation and the variance of $N$, and $\M N$ is the smallest median of $N$.
\end{thm}

One might think that the study of subgraph counts needs to be restricted to finite graphs to ensure that all occuring variables are finite. We stress that this is \emph{not} the case. There are Poisson processes in $\R^d$, such that the associated random geometric graph has a.s. infinitely many vertices but still a.s. a finite number of edges. Similarly, one can also have a.s. infinitely many edges, but still a finite number of triangles. This phenomenon seems to be quite unexplored so far. In the context of concentration properties, a natural question is whether concentration inequalities also hold in these situations. We emphasize that Theorem~\ref{main} only requires $N < \infty$ almost surely and hence covers such cases, as opposed to previous results from \cite{RST_2013, LR_2015} where only finite intensity measures are considered.
\smallskip

The asymptotic exponents in the upper tail bounds for both the expectation and the median are equal to $1/k$. This is actually best possible as will be pointed out later. Note also that the asymptotic exponent of the estimates in previous results from \cite{RST_2013, LR_2015} are $1/k$ for both the upper and the lower tail which is compatible with our upper tail bounds but worse than our lower tail inequalities.
\smallskip

The essential tools for proving our concentration inequalities are new results on Poisson U-statistics. Given a (simple) Poisson point process $\eta$ on $\R^d$, denote by $\eta_{\neq}^k$ the set of all $k$-tuples of distinct elements of $\eta$. A \emph{Poisson U-statistic} of order $k$ is a random variable that can be written as 
$$\sum_{ \xb\in\eta_{\neq}^k} f( \xb), $$ where $f:(\R^d)^k\to\R$ is a symmetric measurable map. The methods that we use in the present work to prove concentration of Poisson U-statistics around the mean are based on tools recently developed by Bachmann and Peccati in \cite{BP_2015}. In this work, concentration inequalities for the edge count in random geometric graphs were established. The approach to achieve these estimates will be generalized in the present article to arbitrary subgraph counts. 
\smallskip

In addition to this, we will present a refinement of the method suggested in \cite{RST_2013, LR_2015} by  Lachi{\`e}ze-Rey, Reitzner, Schulte and Th{\"a}le  which gives concentration of Poisson U-statistics around the median. We improve this method further with the particular advantage, apart from giving clean and easy to use tail bounds, that the underlying Poisson process may have infinite intensity measure.
\smallskip

As an application of the concentration inequalities in Theorem \ref{main} we are going to establish strong laws of large numbers for subgraph counts. In \cite[Chapter 3]{P_2003} one can find strong laws for subgraph counts associated to i.i.d. points in $\R^d$. To the best of our knowledge, there are no comparable strong laws so far for the Poisson process case and a particular feature of our results is that they even apply to certain Poisson processes with non-finite intensity measure.
\smallskip

In recent years, the study of random geometric simplicial complexes built on random point sets in $\R^d$ has attracted a considerable attention, see e.g. \cite{K_2011, DFR_2014, KM_2013, YSA_2014} and in particular the survey \cite{BK_2014} by Bobrowski and Kahle on this topic. Motivations to study these complexes arise particularly from topological data analysis (for a survey on this see \cite{C_2009}), but also from other applications of geometric complexes like sensor networks (see e.g. \cite{SG_2007}). One of the models for random simplicial complexes is the so-called Vietoris-Rips complex. We stress that the number of $k$-dimensional simplices in this complex is exactly given by the number of complete subgraphs on $k$-vertices in the corresponding random geometric graph. So the results in the present paper are directly related to the study of random Vietoris-Rips complexes and may be useful in future research on these models.
\smallskip

Related to our investigations are the findings by  Eichelsbacher, Rai{\v{c}} and Schreiber in \cite{ERS_2015} where deviation inequalities for stabilizing functionals of finite intensity measure Poisson processes were derived. These results are in principle also applicable to U-statistics (and hence to subgraph counts), as it was pointed out in \cite{RST_2013}, although the appearing asymptotic exponents tend to be non-optimal. Moreover, since the constants in the tail bounds from \cite{ERS_2015} depend on the intensity and are not given explicitly, it remains as an interesting open question whether these estimates can be extended to settings with non-finite intensity measures. 
\smallskip

The paper is organized as follows. In Section 2 we will introduce the general assumptions and definitions that form the framework for the entire work. As anticipated, we will prove the announced concentration inequalities not only for subgraph counts, but for local Poisson U-statistics that satisfy certain additional assumptions. The tail estimates for subgraph counts will then follow directly from the more general results. In Section 3 we will present these general concentration inequalities. The announced applications to subgraph counts in random geometric graphs are presented in Section 4. Here, after pointing out how the general results apply to the setting of subgraph counts, we will analyse the asymptotic behaviour of expectation, median and variance which is needed for the subsequent presentation of the strong laws of large numbers. The proofs of all statements are gathered in Section 5.

\section{Framework} \label{s:framework}

For the remainder of the article we denote by $\Nb$ the space of locally finite point configurations in $\R^d$. Elements in $\Nb$ can be regarded as locally finite subsets of $\R^d$, but also as locally finite simple point measures on $\R^d$. In this spirit, for $\xi\in\Nb$, we will use set related notations like for example $\xi\cap A$ as well as measure related notations like $\xi(A)$. Moreover, we denote by $\Nc$ the $\sigma$-algebra over $\Nb$ that is generated by the maps
\begin{align*}
\Nb\to\N\cup\{\infty\}, \xi \mapsto \xi(A),
\end{align*}
where $A$ ranges over all Borel subsets of $\R^d$. Throughout, we will consider a (non-trivial) Poisson point process $\eta $ on $\R^d$. Then $\eta $ is a random element in $\Nb$ and its intensity measure $\mu$ is the measure on $\R^d$ defined by $\mu(A) = \E\eta (A)$ for any Borel set $A\subseteq\R^d$. It will be assumed that $\mu$ is locally finite and does not have atoms. A \emph{(Poisson) U-statistic of order $k$} is a functional $F:\Nb\to\R\cup\{\pm\infty\}$ that can be written as
\begin{align*}
F(\xi) = \sum_{ \xb\in\xi_{\neq}^k} f( \xb),
\end{align*}
where $f:(\R^d)^k \to \R$ is a symmetric measurable map called the \emph{kernel} of $F$, and for any $\xi\in\Nb$,
\begin{align*}
\xi_{\neq}^k = \{(x_1,\ldots,x_k)\in\xi^k : x_i\neq x_j \ \text{whenever} \ i\neq j\}.
\end{align*}
The map $F$ and the corresponding random variable $F(\eta )$ will be identified in what follows if there is no risk of ambiguity. Note that in this work we will only consider U-statistics $F$ with non-negative kernels. Also, it will be assumed throughout that $\eta$ guarantees almost surely $F(\eta)<\infty$ if not stated otherwise.

Motivated by the application to random geometric graphs, we will define further properties that a U-statistic $F$ with kernel  $f\geq 0$ might satisfy. We will refer to these properties in the remainder of the article.
\begin{itemize}
\item [(K1)] There is a constant $\rho_F>0$ such that 
\begin{align*}
f(x_1,\ldots,x_k) > 0 \ \text{whenever} \ \diam(x_1,\ldots,x_k) \leq \rho_F.
\end{align*} 

\item [(K2)] There is a constant $\Theta_F\geq 1$ such that 
\begin{align*}
f(x_1,\ldots,x_k)=0 \ \text{whenever} \ \diam(x_1,\ldots,x_k) > \Theta_F\rho_F.
\end{align*}

\item [(K3)] There are constants $M_F \geq m_F > 0$ such that
\begin{align*}
M_F\geq f(x_1,\ldots,x_k) \geq m_F \ \text{whenever} \ f(x_1,\ldots,x_k)>0.
\end{align*}
\end{itemize}

U-Statistics that satisfy property (K2) are also referred to as \emph{local} U-Statistics. Property (K3) is particularly satisfied for U-Statistics that count occurrences of certain subconfigurations.
\medskip

\section{General Results}
The concentration inequalities for U-statistics we are about to present are based on methods developed in \cite{RST_2013, LR_2015} and \cite{BP_2015}. We begin by citing the results that we use from these articles. To do so, we first need to introduce a further notion. The \emph{local version} of a U-statistic $F$ is defined for any $\xi\in\Nb$ and $x\in\xi$ by
\begin{align*}
 F(x,\xi) = \sum_{\yb\in (\xi\setminus x)^{k-1}_{\neq}} f(x,\yb),
\end{align*}
where $\xi\setminus x$ is shorthand for $\xi\setminus\{x\}$. Note, that
\begin{align*}
 F(\xi) = \sum_{x\in\xi}F(x,\xi).
\end{align*}

The upcoming result uses a notion introduced in \cite{BP_2015}: A U-statistic $F$ is called \emph{well-behaved} if there is a measurable set $B\subseteq \Nb$ with $\P(\eta\in B)=1$, such that
\begin{enumerate}
\item $F(\xi)<\infty$ for all $\xi\in B$,
\item $\xi\cup\{x\}\in B$ whenever $\xi\in B$ and $x\in\R^d$,
\item $\xi\setminus \{x\}\in B$ whenever $\xi\in B$ and $x\in\xi$.
\end{enumerate}
Roughly speaking, the motivation for this notion is that for a well-behaved \linebreak U-statistic $F$ one has almost surely $F(\eta)<\infty$ as well as $F(\eta\cup\{x\})<\infty$ for $x\in\R^d$ and $F(\eta\setminus\{x\})<\infty$ for $x\in\eta$. This ensures that the local versions $F(x,\eta)$ for $x\in\eta$ and also $F(x, \eta \cup\{x\})$ for $x\in\R^d$ are almost surely finite, thus preventing technical problems arising from non-finiteness of these quantities. The following theorem combines \cite[Corollary 5.2]{BP_2015} and \cite[Corollary 5.3]{BP_2015} and the fact that by \linebreak \cite[Lemma 3.4]{Ka_2002} the exponential decay of the upper tail implies existence of all moments of $F$. 

\begin{thm} \label{generalUpperThm}
Let $F$ be a well-behaved U-statistic of order $k$ with kernel $f\geq 0$. Assume that for some $\alpha\in[0,2)$ and $c>0$ we have almost surely
\begin{align} \label{condUstat}
\sum_{x\in\eta }F(x,\eta )^2 \leq c F^\alpha.
\end{align}
Then all moments of $F$ exist and for all $r\geq 0$,
\begin{align*}
\P(F\geq \E F + r) \leq \exp\left(-\frac{((\E F + r)^{1-\alpha/2}-(\E F)^{1-\alpha/2})^2}{2k^2c}\right)
\end{align*}
and 
\begin{align*}
\P(F\leq \E F - r) \leq \exp\left(-\frac{r^2}{2k\V F}\right)  .
\end{align*}
\end{thm}

Note that condition (K2), as defined in Section \ref{s:framework}, clearly ensures that all U-statistics considered in the present paper are well-behaved in the sense described above.

In addition to the latter result, a variation of the approach presented in \cite{RST_2013, LR_2015} gives that for finite intensity measure Poisson processes, the condition (\ref{condUstat}) also implies a concentration inequality for the median instead of the expectation of the considered U-statistic. Moreover, it is possible to extend these tail estimates to U-statistics built over non-finite intensity measure processes, resulting in the forthcoming theorem. To state this result, we first need to introduce a further notation. For any real random variable $Z$, we denote by $\M Z$ the smallest median of $Z$, i.e.
\begin{align} \label{MedDef}
\M Z = \inf\{x\in\R: \P(Z\leq x)\geq 1/2\}.
\end{align}
Note that $\M Z$ is exactly the value which the quantile-function of $Z$ takes at $1/2$. We are well prepared to state the announced result.

\begin{thm} \label{thm2}
Let $F$ be a U-statistic of order $k$ with kernel $f\geq 0$. Assume that $F$ is almost surely finite and satisfies (\ref{condUstat}) for some $\alpha\in[0,2)$ and $c>0$. Then for all $r\geq 0$,
\begin{align*}
\P(F> \M F+r) &\leq 2 \exp\left(-\frac{r^2}{4k^2c(r+\M F)^{\alpha}}\right),\\
\P(F< \M F-r) &\leq 2 \exp\left(-\frac{r^2}{4k^2c(\M F)^{\alpha}}\right). 
\end{align*}
\end{thm}

In the light of the above results, a natural approach towards concentration inequalities for U-statistics is to establish condition (\ref{condUstat}). For U-statistics satisfying the conditions (K1) to (K3) we obtain the following.

\begin{thm} \label{lem2}
Let $F$ be a U-statistic of order $k$ with kernel $f\geq 0$ that satisfies (K1) to (K3). Then for any $\xi\in\Nb$ we have
\begin{align} \label{condUStat2}
 \sum_{x\in \xi} F(x,\xi)^2 \leq c_dF(\xi)^{\tfrac{2k-1}{k}},
\end{align}
where
\[
c_d = \left(2\left\lceil \Theta_F\sqrt{d}\right\rceil+1\right)^{2d(k-1)}\cdot M_F^2 \left(\frac{k^k}{m_Fk!}\right)^{\tfrac{2k-1}{k}}.
\]
\end{thm}
The above statement guarantees that Theorem \ref{generalUpperThm} and Theorem \ref{thm2} apply to any almost surely finite U-statistic $F$ that satisfies (K1) to (K3). Thus, the following tail estimates are established.

\pagebreak

\begin{cor} \label{concUstat}
Let $F$ be a U-statistic of order $k$ with kernel $f\geq 0$. Assume that $F$ is almost surely finite and satisfies (K1) to (K3). Then all moments of $F$ exist and for all $r\geq 0$,
\begin{align*}
\P(F\geq \E F + r) &\leq \exp\left(-\frac{((\E F + r)^{1/(2k)}-(\E F)^{1/(2k)})^2}{2k^2c_d}\right),\\
\P(F\leq \E F - r) &\leq \exp\left(-\frac{r^2}{2k\V F}\right),\\
\P(F> \M F+r) &\leq 2 \exp\left(-\frac{r^2}{4k^2c_d(r+\M F)^{2-1/k}}\right),\\
\P(F< \M F-r) &\leq 2 \exp\left(-\frac{r^2}{4k^2c_d(\M F)^{2-1/k}}\right),
\end{align*}
where $c_d$ is defined as in Theorem \ref{lem2}.
\end{cor}

We conclude this section with a brief discussion about optimality of the upper tail concentration inequalities that were established above. Consider a U-statistic $F$ such that the assumptions of Corollary \ref{concUstat} hold. Assume that
\begin{align}\label{tailCond}
\P(F\geq M + r) \leq \exp(-I(r)),
\end{align}
where $M$ is either the mean or a median of $F$ and $I$ is a function that satisfies
\begin{align}\label{ICond}
\liminf_{r\to\infty} I(r)/r^a > 0
\end{align}
for some $a>0$. The upper tail estimates from Corollary \ref{concUstat} yield such functions $I$ for the exponent $a=1/k$. The next result states that this is optimal.
\begin{prop}\label{optProp}
Let $F$ be a U-statistic of order $k$ with positive kernel $f\geq 0$ such that $F$ satisfies (K1) to (K3). Let $I$ and $a>0$ be such that (\ref{tailCond}) and (\ref{ICond}) are verified. Then $a\leq 1/k$.
\end{prop}

\section{Subgraph Counts in Random Geometric Graphs}\label{s:SCRGG}

In the following, we are going to investigate a model for random geometric graphs which was particularly investigated in \cite{LP_2013_1} and \cite{LP_2013_2}. Let $S\subseteq \R^d$ be such that $S=-S$. To any countable subset $\xi\subset \R^d$ we assign the \emph{geometric graph} $\mathfrak{G}_S(\xi)$ with vertex set $\xi$ and an edge between two distinct vertices $x$ and $y$ whenever $x-y\in S$. For $ \xb=(x_1,\ldots, x_k)\in(\R^d)^k$ we will occasionally write $\mathfrak{G}_S( \xb)$ instead of $\mathfrak{G}_S(\{x_1,\ldots,x_k\})$. Now, assuming that the vertices are chosen at random according to some Poisson point process $\eta $, we obtain the \emph{random geometric graph} $\mathfrak{G}_S(\eta )$. Denote the closed ball centered at $x\in\R^d$ with radius $\rho\in\R_+$ by $B(x,\rho)$. Throughout, we will assume that $B(0,\rho) \subseteq S \subseteq B(0,\theta\rho)$ for some $\rho>0$ and $\theta\geq 1$. Note that if we take $\theta = 1$, then $S = B(0,\rho)$ and we end up with the classical model of random geometric graphs for the euclidean norm, often referred to as \emph{random disk graph}. Also, the classical geometric graphs based on any other norm in $\R^d$ are covered by the model introduced above. These classical models are extensively described in \cite{P_2003} for the case when the underlying point process $\eta$ has finite intensity measure.

Before proceeding with the discussion, we present a picture that illustrates how the graphs that will be considered in the following might look like (in a window around the origin).
\medskip

\begin{figure}[h]
\includegraphics[scale=0.24]{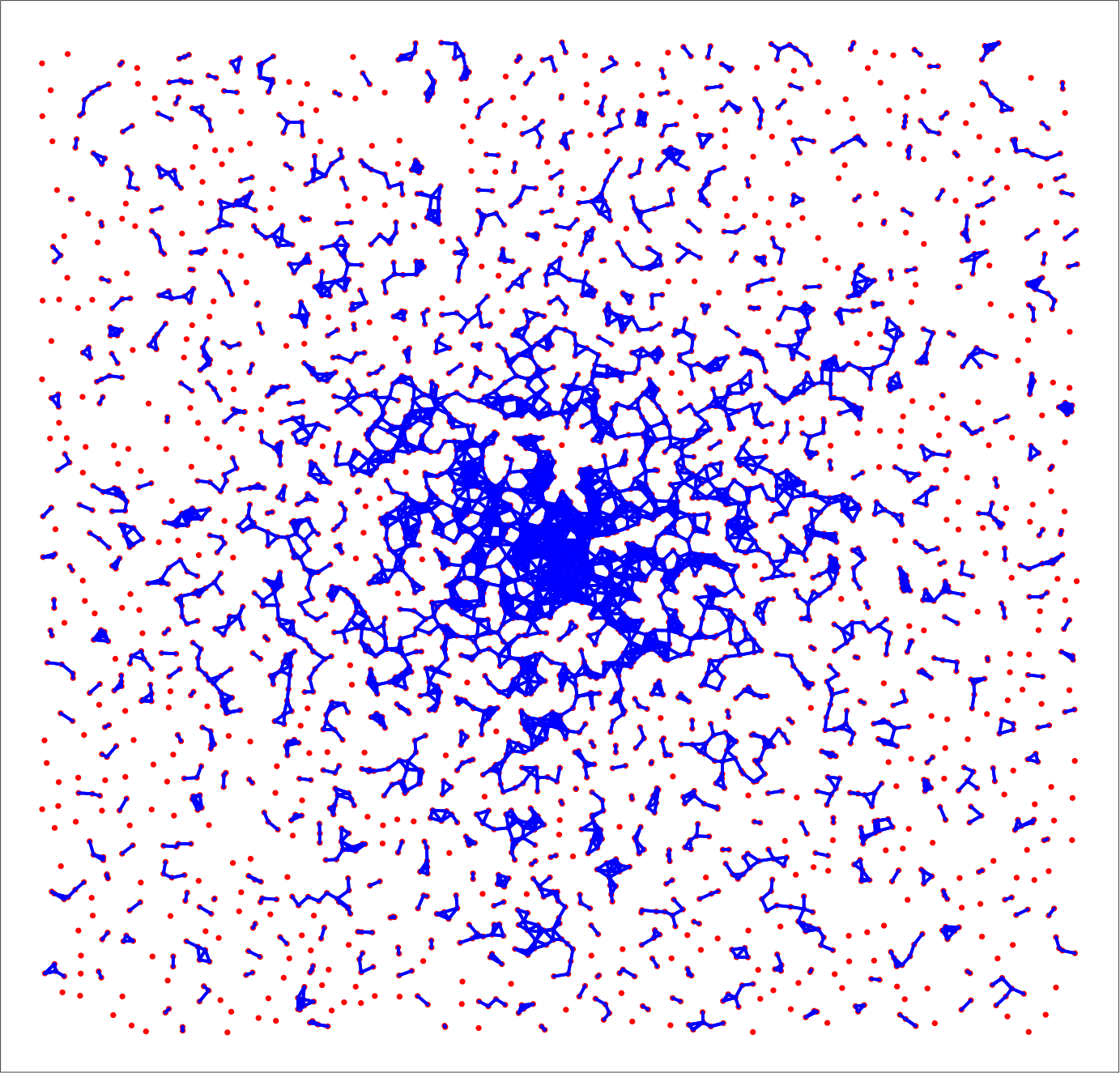}
\caption{Random unit disk graph, intensity measure $\mu=18(\lVert x \rVert+1)^ {-1} dx$}
\label{RGGSim}
\end{figure}

A class of random variables that is frequently studied in the literature on random geometric graphs are the subgraph counts. For any connected graph $H$ on $k$ vertices the corresponding \emph{subgraph count} is the U-statistic
\[
{N^H}(\eta ) = \sum_{ \xb\in \eta ^k_{\neq}} f_H( \xb),
\]
where the kernel $f_H$ is given by
\[
f_H( \xb) = \frac{|\{\text{subgraphs} \  H' \ \text{of} \ \mathfrak{G}_S( \xb): H'\cong H\}|}{k!}.
\]
At this, a subgraph $H'$ of $\mathfrak{G}_S( \xb)$ is a graph on a vertex set which is a subset of $\{x_1,\ldots,x_k\}$ such that any edge of $H'$ is also an edge of $\mathfrak{G}_S( \xb)$. Note that we write $H'\cong H$ if $H'$
and $H$ are isomorphic graphs. So, the random variable ${N^H}(\eta )$ counts the occurrences of $H$ as a subgraph of $\mathfrak{G}_S(\eta )$. We will write $N = N^H$ if $H$ is clear form the context. Note that we do not consider induced subgraphs here. Also note that since we are particularly interested in graphs built on non-finite intensity measure Poisson processes, it is possible that $N$ takes the value $\infty$.

\subsection{Concentration Inequalities for Subgraph Counts}
From the general concentration result Corollary \ref{concUstat} we obtain the deviation inequalities for subgraph counts stated in Theorem \ref{main}. To see this, let $H$ be a connected graph on $k$ vertices and consider the corresponding subgraph count $N^H$. Denote by $a_H$ the number of subgraphs isomorphic to $H$ in a complete graph on $k$ vertices. Moreover, let $\diam(H)$ be the diameter of $H$, i.e. the length of the longest shortest path in $H$. Then $N=N^H$ satisfies the conditions (K1) to (K3) with
\begin{align*}
\rho_N = \rho, \ \ \Theta_N = \diam(H)\theta,\ \ M_N  = \frac{a_H}{k!},\ \ m_N = \frac{1}{k!}.
\end{align*}
Hence, Theorem \ref{main} follows from Corollary \ref{concUstat} where
\begin{align*}
c_d = \left(2\left\lceil\diam(H)\theta\sqrt{d}\right\rceil+1\right)^{2d(k-1)}\cdot \left(\frac{a_H}{k!}\right)^2 k^{2k-1}.
\end{align*}

\subsection{Asymptotic Behaviour of Subgraph Counts}\label{ABSC}
The concentration inequalities presented in our main Theorem \ref{main} depend on expectation, median or variance of the considered subgraph count. It is therefore of interest, how these quantities behave asymptotically in settings where the parameters of the model are being varied. In what comes, we will study these asymptotic behaviours for sequences of subgraph counts with the particular goal of establishing strong laws of large numbers.

For the remainder of this section, we consider a sequence of random geometric graphs $(\mathfrak{G}_{\rho_tS}(\eta _t))_{t\in\N}$. At this, $(\eta _t)_{t\in \N}$ is a sequence of Poisson point processes in $\R^d$ where each $\eta _t$ has intensity measure $\mu_t = t \mu$ and $\mu$ is given by a Lebesgue density $m$. Moreover, $(\rho_t)_{t\in\N}$ is a sequence of positive real numbers such that $\rho_t\to 0$ as $t\to\infty$. In correspondence with the conventions at the beginning of Section \ref{s:SCRGG}, the set $S$ is assumed to satisfy $B(0,1)\subseteq S \subseteq B(0,\theta)$ for some $\theta\geq 1$. Any connected graph $H$ on $k$ vertices now yields a corresponding sequence $N^H_t = N_t$ of subgraph counts in the random graphs $\mathfrak{G}_{\rho_t S}(\eta _t)$.

We are particularly interested in settings where the underlying Poisson point process has non-finite intensity measure. In this situation it is not automatically guaranteed that the considered subgraph counts are almost surely finite. Note (again) that by Theorem \ref{main}, almost sure finiteness of a subgraph count is equivalent to existence of all its moments. The upcoming result gives a sufficient condition for almost sure finiteness of the considered random variables.

\begin{prop}\label{IntegrableProp}
Assume that the Lebesgue density $m$ is bounded. Assume in addition that there exists a constant $c\geq 1$ such that
\begin{align} \label{desityCond}
m(x) \leq c m(y) \ \ \text{whenever} \ \ \lVert x-y \rVert \leq \diam(H)\theta\sup_{t\in\N}\rho_t.
\end{align}
Then, for any $t\in\N$, the random variable $N_t$ is almost surely finite if and only if
\begin{align} \label{integrableCond}
\int_{\R^d} m(x)^k dx < \infty.
\end{align}
\end{prop}

\begin{rem*}
The reason for considering densities that satisfy condition (\ref{desityCond}) is that this allows us to use a dominated convergence argument in the proof of the upcoming Theorem \ref{SGCExp}. A class of densities that satisfy condition (\ref{desityCond}) is given by
\begin{align*}
m(x) = A (\lVert x \rVert + 1)^{-\gamma},
\end{align*}
where $A,\gamma>0$. Indeed, for any $x,y\in \R^d$,
\begin{align*}
\frac{m(x)}{m(y)} = \left(\frac{\lVert y \rVert + 1}{\lVert x \rVert + 1}\right)^\gamma \leq \left(\frac{\lVert x \rVert + \lVert x-y \rVert + 1}{\lVert x \rVert + 1}\right)^\gamma \leq (1+\lVert x-y \rVert)^\gamma.
\end{align*}
An example of how the resulting random graphs for these densities might look like is illustrated in Figure \ref{RGGSim} above. Observe also that condition (\ref{integrableCond}) for a density $m$ given by $m(x)=A(\lVert x \rVert + 1)^{-\gamma}$ is equivalent to $k > d / \gamma$. In particular, the graph in Figure \ref{RGGSim} has almost surely infinitely many edges, but any subgraph count of order at least $3$ (for example the number of triangles) is almost surely finite.
\end{rem*}

Our goal is to establish strong laws for suitably rescaled versions of the subgraph counts $N_t$. To do so, it is crucial to know the asymptotic behaviour of the expectation $\E N_t$ and the variance $\V N_t$ as $t\to\infty$. In the case when the intensity measure of the underlying Poisson point process is finite, these asymptotics are well known (see e.g. \cite{P_2003}). Since our focus is on random geometric graphs built on non-finite intensity measure processes, the first step towards establishing strong laws is to study expectation and variance asymptotics. We also have concentration inequalities with respect to the median, so the asymptotic behaviour of the median may as well be of interest in applications. The upcoming result addresses these issues. Note that we will write $a_n\sim b_n$ if sequences $(a_n)_{n\in\N}$ and $(b_n)_{n\in\N}$ are \emph{asymptotically equivalent}, meaning that $\lim_{n\to\infty} a_n/b_n = 1$.

\begin{thm} \label{SGCExp}
Assume that the intensity measure $\mu$ is given by a continuous and bounded Lebesgue density $m:\R^d\to\R_+$ that satisfies (\ref{desityCond}) and (\ref{integrableCond}). Then the following holds:
\begin{enumerate}
\item There exists a constant $a>0$ such that
\begin{align*}
\E N_t \sim at^k\rho_t^{d(k-1)}.
\end{align*}
\item If $\lim_{t\to\infty}\E N_t = \infty$, then
\begin{align*}
\M N_t \sim \E N_t.
\end{align*}
\item There exist constants $A^{(n)} > 0$ for $1\leq n \leq k$ such that
\begin{align*}
\V N_t \sim t^k\rho_t^{d(k-1)}\sum_{n=1}^k (t\rho_t^d)^{k-n} A^{(n)}.
\end{align*}
\end{enumerate}
\end{thm}

The concentration inequalities together with asymptotic results for expectation and variance can be used to obtain the upcoming strong laws for subgraph counts. Note that for the next result to hold, we only need that expectation and variance behave asymptotically as stated in Theorem \ref{SGCExp}. According to \cite[Proposition 3.1, Proposition 3.7]{P_2003} this asymptotic behaviour is also in order whenever the intensity measure of $\eta$ is finite and has a bounded Lebesgue density, so the following theorem also applies in these situations. The strong law presented below complements the results \cite[Theorem 3.17, Theorem 3.18]{P_2003} that deal with random geometric graphs built over i.i.d. points.

\begin{thm}\label{SLLN}
Assume that the statements (i) and (iii) of Theorem \ref{SGCExp} hold. \linebreak Assume in addition that for some $\gamma>0$,
\begin{align} \label{regime}
\liminf_{t\to\infty} t^{k-\gamma} \rho_t^{d(k-1)} > 0.
\end{align}
Let $a>0$ denote the limit of the sequence $\E N_t/(t^k\rho_t^{d(k-1)})$. Then
\begin{align*}
\frac{N_t}{t^k \rho_t^{d(k-1)}} \overset{a.s.}{\longrightarrow} a \ \ \text{as} \ \ t\to\infty.
\end{align*}
\end{thm}

\section{Proofs}

\subsection*{Proof of Theorem \ref{lem2}}

To get prepared for the proof of Theorem \ref{lem2}, which is crucial to establish the desired deviation and concentration inequalities, we first prove the following two lemmas.

\begin{lem} \label{lem3}
Let $x_1,\ldots,x_k,y_1,\ldots, y_k\in\R_{\geq 0}$. Then
\[
\prod_{i=1}^k x_i + \prod_{i=1}^k y_i \leq \prod_{i=1}^k \max(x_i,y_i) + \prod_{i=1}^k \min(x_i,y_i).
\]
\end{lem}
\begin{proof}
Let $I$ and $J$ be the sets of indices $i$ satisfying $x_i<y_i$ and $x_i\geq y_i$, respectively. Then we have
\[
\prod_{i\in I} x_i < \prod_{i\in I} y_i \ \ \text{ and } \ \ \prod_{i\in J} x_i \geq \prod_{i\in J} y_i.
\]
It follows that
\begin{align*}
\prod_{i=1}^k x_i + \prod_{i=1}^k y_i &= \prod_{i\in I} x_i \prod_{i\in J} x_i + \prod_{i\in I} y_i \prod_{i\in J} y_i\\ &=
\prod_{i\in I} x_i \prod_{i\in J} x_i + \left(\prod_{i\in I}y_i - \prod_{i\in I} x_i \right) \prod_{i\in J} y_i + \prod_{i\in I} x_i\prod_{i\in J} y_i\\ &\leq
\prod_{i\in I} x_i \prod_{i\in J} x_i + \left(\prod_{i\in I}y_i - \prod_{i\in I} x_i \right) \prod_{i\in J} x_i + \prod_{i\in I} x_i\prod_{i\in J} y_i\\ &=
\prod_{i\in I}y_i\prod_{i\in J} x_i + \prod_{i\in I} x_i\prod_{i\in J} y_i =
\prod_{i=1}^k \max(x_i,y_i) + \prod_{i=1}^k \min(x_i,y_i).
\end{align*}
\end{proof}

Using the lemma above, we can prove the following.

\begin{lem} \label{lem1}
Let $n_1,\ldots,n_N\in \R_{\geq 0}$, $k\geq 2$ and let $\pi_1,\ldots,\pi_k$ be permutations of $\{1,\ldots, N\}$. Then
 \[
  \sum_{i=1}^N \prod_{j=1}^k n_{\pi_j(i)} \leq \sum_{i=1}^N n_i^k.
 \]
\end{lem}

\begin{proof}
Without loss of generality we can assume
\begin{align} \label{wlog1}
n_1\leq n_2\leq\ldots\leq n_N.
\end{align}
The proof is by induction on $N$. For $N=1$ there is nothing to prove, for $N=2$ the result follows from Lemma \ref{lem3}. So let $N>2$ and assume the result holds for $N-1$. Let
\[
t = |\{i: \pi_j(i)=N \text{ for some } j\}|.
\]
We prove by induction on $t$ that the following holds: There exist permutations $\tilde{\pi}_1,\ldots,\tilde{\pi}_k$ of $\{1,\ldots,N-1\}$ such that
\[
\sum_{i=1}^{N} \prod_{j=1}^k n_{\pi_j(i)} \leq \sum_{i=1}^{N-1} \prod_{j=1}^k n_{\tilde{\pi}_j(i)} + n_N^k.
\]
The result then follows by applying the induction hypotheses (with respect to $N$) to the right hand side of this inequality.

For $t=1$ the permutations $\tilde{\pi}_i$ obviously exist since in this case there is an index $l$ such that $\pi_j(l) = N$ for all $j$. Consider the case $t>1$. Then there exist indices $l_1\neq l_2$ such that $\pi_{j_1}(l_1)=\pi_{j_2}(l_2) = N$ for some $j_1,j_2$. We define permutations $\bar{\pi}_1,\ldots,\bar{\pi}_k$ of $\{1,\ldots,N\}$ by $\bar{\pi}_j = \pi_j$ if $\pi_j(l_1) > \pi_j(l_2)$ and by
\[
\bar{\pi}_j(i) =
\begin{cases}
\pi_j(i) & \text{if } i\neq l_1,l_2\\
\pi_j(l_2) & \text{if } i= l_1\\
\pi_j(l_1) & \text{if } i= l_2\\
\end{cases}
\]
if $\pi_j(l_1) < \pi_j(l_2)$. Then, using Lemma \ref{lem3} together with \eqref{wlog1} we obtain
\begin{align*}
&\sum_{i=1}^{N} \prod_{j=1}^k n_{\pi_j(i)} 
= \sum_{i=1,i\neq l_1,l_2}^{N} \prod_{j=1}^k n_{\pi_j(i)} + \prod_{j=1}^k n_{\pi_j(l_1)} + \prod_{j=1}^k n_{\pi_j(l_2)}\\
&\leq \sum_{i=1,i\neq l_1,l_2}^{N} \prod_{j=1}^k n_{\pi_j(i)} + \prod_{j=1}^k \max(n_{\pi_j(l_1)}, n_{\pi_j(l_2)}) + \prod_{j=1}^k \min(n_{\pi_j(l_1)}, n_{\pi_j(l_2)})\\
&= \sum_{i=1}^{N} \prod_{j=1}^k n_{\bar{\pi}_j(i)}.
\end{align*}
Note, that $\bar{\pi}_j(l_2)\neq N$ for all $j$ and hence $|\{i:\bar{\pi}_j(i)=N \text{ for some } j\}| = t-1$. Applying the induction hypotheses (with respect to $t$) yields the existence of the desired permutations $\tilde{\pi}_j$. This concludes the proof.
\end{proof}

\begin{proof}[Proof of Theorem \ref{lem2}]

First note that the result holds trivially in the case when $F(\xi) = \infty$. Moreover, if $F(\xi)<\infty$, then property (K3) guarantees that $F(x,\xi)=0$ for all but finitely many $x\in\xi$. Hence, we can assume without loss of generality that $\xi$ is finite and that $F(x,\xi)>0$ for all $x\in\xi$.

We consider a tiling of $\R^d$ into cubes of diagonal $\rho_F$. Since we assumed that $\xi$ is finite, we can choose out of these cubes $C_1,\ldots,C_N$ such that $\xi\subseteq \cup C_i$. Then for
any distinct $x_1\ldots,x_k\in\xi$ contained in the same cube we have $f(x_1\ldots,x_k)\geq m_F$. Moreover, translating the tiling a bit if necessary, we can assume without loss of generality that each point $x\in\xi$ is contained in the interior of exactly one of the cubes.

Observe that condition (K2) implies the following: Let $i\in\{1,\ldots, N\}$. Then there are $(2\lceil\sqrt{d} \Theta_F\rceil + 1)^d =: q$ many cubes $C^1_i,\ldots,C^q_i$ such that for any $x\in C_i$ we have $\{y_1,\ldots,y_{k-1}\}\subset \cup_{j=1}^qC_i^j$ whenever $f(x,y_1,\ldots,y_{k-1}) > 0$. Note also, that the $C_i^j$ can be chosen such that
for any fixed $j$ it holds that $\{C_1,\ldots, C_N\} = \{C_1^j,\ldots,C_N^j\}$.

Let $[q] := \{1,\ldots,q\}$ and for any $x\in\xi\cap C_i$ and $(j_1,\ldots,j_{k-1})\in[q]^{k-1}$ let
\[
C_i(x,j_1,\ldots,j_{k-1}) := \left[(C_i^{j_1}\cap(\xi\setminus x))\times \ldots \times (C_i^{j_{k-1}}\cap(\xi\setminus x))\right]_{\neq}.
\]
Then we have
\begin{align*}
& \sum_{x\in\xi}F(x,\xi)^2 = \sum_{i=1}^N \sum_{x\in \xi\cap C_i}F(x,\xi)^2 \\ &= \sum_{i=1}^N
\sum_{x\in \xi\cap C_i}\left(\sum_{(j_1,\ldots,j_{k-1})\in[q]^{k-1}}\sum_{\yb\in C_i(x,j_1,\ldots,j_{k-1})} f(x,\yb)\right)^2,
\end{align*}
where $\yb = (y_1,\ldots,y_{k-1})$. Now let $n_i = |\xi\cap C_i|$ and $n_i^{(j)} = |\xi\cap C_i^j|$. Then by condition (K3) the last expression in the above display does not exceed
\[
 \sum_{i=1}^N
\sum_{x\in \xi\cap Ci}\left(\sum_{(j_1,\ldots,j_{k-1})\in[q]^{k-1}} M_F\prod_{l=1}^{k-1} n_i^{(j_l)}\right)^2.
\]
Moreover, we have
\[
\left(\sum_{(j_1,\ldots,j_{k-1})\in[q]^{k-1}} M_F\prod_{l=1}^{k-1} n_i^{(j_l)}\right)^2 \leq q^{k-1}M_F^2 \sum_{(j_1,\ldots,j_{k-1})\in[q]^{k-1}}\prod_{l=1}^{k-1} (n_i^{(j_l)})^2.
\]
Thus, it follows from the above considerations that
\[
\sum_{x\in\xi}F(x,\xi)^2
\leq q^{k-1} M_F^2\sum_{i=1}^N
\sum_{x\in \xi\cap C_i}\sum_{(j_1,\ldots,j_{k-1})\in[q]^{k-1}}\prod_{l=1}^{k-1} (n_i^{(j_l)})^2. 
\]
Rearranging the triple sum, we can write the right hand side as
\begin{align*}
 & q^{k-1} M_F^2\sum_{(j_1,\ldots,j_{k-1})\in[q]^{k-1}}\sum_{i=1}^N
\sum_{x\in \xi\cap C_i}\prod_{l=1}^{k-1} (n_i^{(j_l)})^2\\ =  & \ q^{k-1} M_F^2\sum_{(j_1,\ldots, j_{k-1})\in[q]^{k-1}}\sum_{i=1}^N
n_i\prod_{l=1}^{k-1} (n_i^{(j_l)})^2.
\end{align*}
By Lemma \ref{lem1} this expression does not exceed
\begin{align*}
 q^{k-1} M_F^2\sum_{(j_1,\ldots,j_{k-1})\in[q]^{k-1}}\sum_{i=1}^N
n_i^{2k-1} &= q^{2(k-1)} M_F^2\sum_{i=1}^N
n_i^{2k-1}\\ &\leq q^{2(k-1)} M_F^2\left(\sum_{i=1}^N
n_i^k\right)^{\tfrac{2k-1}{k}}.
\end{align*}
At this, the latter inequality holds by monotonicity of the $p$-Norm.

It remains to prove
\[
 \sum_{i=1}^N n_i^k \leq \frac{k^{k-1}}{m_F(k-1)!}\sum_{x\in\xi} F(x,\xi).
\]
We will see that for any $x\in\xi\cap C_i$ we have
\begin{align} \label{ineq1}
n_i^{k-1}\frac{(k-1)!}{k^{k-1}}\leq \frac{F(x,\xi)}{m_F}.
\end{align}
This then
gives, as desired, that
\begin{align*}
\frac{k^{k-1}}{(k-1)!}\sum_{x\in\xi} F(x,\xi) &= m_F\sum_{i=1}^N\sum_{x\in\xi\cap C_i} \frac{k^{k-1}F(x,\xi)}{m_F(k-1)!}\\ &\geq
m_F\sum_{i=1}^N\sum_{x\in\xi\cap C_i} n_i^{k-1} = m_F\sum_{i=1}^N n_i^k.
\end{align*}
To prove \eqref{ineq1}, let $x\in\xi\cap C_i$. Consider the set
\[
A := \{(y_1,\ldots,y_{k-1})\in(\xi\setminus x)
^{k-1}_{\neq} \ : \ f(x,y_1,\ldots,y_{k-1}) > 0\}.
\]
Then, by definition of $F(x,\xi)$ and by condition (K3) we have
\[
 F(x,\xi) = \sum_{(y_1,\ldots,y_{k-1})\in A} f(x,y_1,\ldots,y_{k-1})\geq \sum_{(y_1,\ldots,y_{k-1})\in A} m_F.
 \]
Thus
\[
 \frac{F(x,\xi)}{m_F} \geq |A|.
\]
Since by condition (K1) we have $((\xi\setminus x)\cap C_i)^{k-1}_{\neq} \subseteq A$, it follows that
\[
\prod_{t=1}^{k-1} (n_i - t) =  \prod_{t=0}^{k-2} (n_i - 1 - t) = |((\xi\setminus x)\cap C_i)^{k-1}_{\neq}| \leq \frac{F(x,\xi)}{m_F}.
\]
Moreover, it is straightforward to check that
\[
\frac{\prod_{t=1}^{k-1} (n_i - t)}{n_i^{k-1}} \geq \frac{(k-1)!}{k^ {k-1}}
\]
whenever $n_i>k$. Hence, it follows that for $n_i>k$ we have
\[
n_i^{k-1} \frac{(k-1)!}{k^{k-1}} \leq \prod_{t=1}^{k-1} (n_i - t)\leq \frac{F(x,\xi)}{m_F}.
\]
Therefore, the inequality in (\ref{ineq1}) holds for $n_i>k$. To conclude that the inequality also holds for $n_i\leq k$, recall that we assumed whithout loss of gererality that $F(y,\xi)>0$ for all $y\in\xi$, thus $F(x,\xi)>0$. Hence, it follows from condition (K3) together with symmetry of $f$ that even $F(x,\xi)\geq m_F(k-1)!$. Thus, for $n_i\leq k$, we obtain
\[
n_i^{k-1} \frac{(k-1)!}{k^{k-1}} \leq k^{k-1} \frac{(k-1)!}{k^{k-1}} = (k-1)! \leq \frac{F(x,\xi)}{m_F}.
\]
This concludes the proof.
\end{proof}

\subsection*{Proof of Theorem \ref{thm2}}
The approach that is described in the following to obtain the deviation inequalities presented in Theorem \ref{thm2} is a refinement of the method suggested in \cite{RST_2013, LR_2015}. We will use the convex distance for Poisson point processes which was introduced in \cite{R_2013}. Let $\Nb_{\rm fin}$ be the space of finite point configurations in $\R^d$, equipped with the $\sigma$-algebra $\mathcal{N}_{\rm fin}$ that is obtained by restricting $\mathcal{N}$ to $\Nb_{\rm fin}$. Then for any $\xi\in \Nb_{\rm fin}$ and $A\in \mathcal{N}_{\rm fin}$ this distance is given by
\[
d_T(\xi, A) = \max_{u \in S(\xi)}\min_{\delta\in A} \sum_{x\in\xi\setminus\delta}u(x),
\]
where
\[
S(\xi) = \{u: \R^d\to\R_{\geq 0} \text{ measurable with} \ \sum_{x\in\xi}u(x)^2\leq 1\}.
\]
To obtain the deviation inequalities for a U-statistic $F$, we will first relate $F$ in a reasonable way to $d_T$. Then we will use the inequality
\begin{align} \label{concentration}
\P(\eta \in A)\P(d_T(\eta , A)\geq s)\leq \exp\left(-\frac{s^2}{4}\right) \ \ \ \ \text{for} \ A\in \mathcal{N}_{\rm fin}, \ s\geq 0,
\end{align}
which was proved in \cite{R_2013}. For the upcoming proof of Theorem \ref{thm2} we also need the following relation, stated in \cite{LR_2015}. We carry out the corresponding straightforward computations for the sake of completeness.

\begin{lem} \label{lem4}
Let $F$ be a U-statistic of order $k$ with kernel $f\geq 0$. Then for any $\xi, \delta\in \Nb_{\rm fin}$ we have
\[
F(\xi)\leq k\sum_{x\in\xi\setminus\delta}F(x,\xi) + F(\delta).
\]
\end{lem}

\begin{proof}
We have
\begin{align*}
F(\xi)
&= 
\sum_{ \xb\in\xi_{\neq}^k} {\ind} (\exists x_i \notin  \delta) f( \xb)
+
\sum_{ \xb\in (\xi \cap \delta)_{\neq}^k} f( \xb)
\\ &\leq 
\sum_{i=1}^k \sum_{ \xb\in\xi_{\neq}^k} {\ind} (x_i \notin  \delta) f( \xb)
+
\sum_{ \xb\in \delta_{\neq}^k} f( \xb)
\\ & =
k \sum_{ \xb\in\xi_{\neq}^k} {\ind} (x_1 \notin  \delta) f( \xb)
+
\sum_{ \xb\in \delta_{\neq}^k} f( \xb)
\\ &= 
k\sum_{x\in\xi\setminus\delta}F(x,\xi) + F(\delta)
\end{align*}
where the third line holds by symmetry of $f$.
\end{proof}

Before we proof Theorem \ref{thm2} in its full generality, we need to establish the corresponding result for finite intensity measure processes. The proof of the next statement is a variation of the method that was suggested in \cite[Sections 5.1 and 5.2]{RST_2013} and \cite[Section 3]{LR_2015}.

\begin{prop}\label{propMedian}
Assume that the intensity measure of $\eta$ is finite. Let $F$ be a U-statistic of order $k$ with kernel $f \geq 0$ and let $\mm$ be a median of $F$. Assume that $F$ satisfies (\ref{condUstat}) for some $\alpha\in[0,2)$ and $c>0$. Then for all $r\geq 0$ one has
\begin{align*}
\P(F\geq \mm+r) &\leq 2 \exp\left(-\frac{r^2}{4k^2c(r+\mm)^{\alpha}}\right),\\
\P(F\leq \mm-r) &\leq 2 \exp\left(-\frac{r^2}{4k^2c \mm^{\alpha}}\right). 
\end{align*}
\end{prop}

\begin{proof}
Let $\xi\in\Nb_{\rm fin}$ and $A\in \mathcal{N}_{\rm fin}$. Define the map $u_\xi:\R^d\to\R_{\geq 0}$ by
\begin{align*}
u_\xi(x) = \frac{F(x,\xi)}{\sqrt{\sum_{y\in\xi}F(y,\xi)^2}} \ \ \text{if} \ \ x\in\xi \ \ \text{and} \ \ u_\xi(x)=0 \ \ \text{if} \ \ x\notin \xi.
\end{align*}
Then we have $u_\xi\in S(\xi)$, thus
\[
d_T(\xi,A)\geq  \min_{\delta\in A} \frac{ \sum_{x\in\xi\setminus\delta}F(x,\xi)}{\sqrt{\sum_{y\in\xi}F(y,\xi)^2}}.
\]
Moreover, by Lemma \ref{lem4}, for any $\delta\in A$ we have
\[
F(\xi)\leq k\sum_{x\in\xi\setminus\delta}F(x,\xi) + F(\delta).
\]
Thus, since by (\ref{condUstat}) we have $\sum_{y\in\xi}F(y,\xi)^2 \leq c F(\xi)^\alpha$ for $\P_\eta$-a.e. $\xi\in\Nb_{\rm fin}$, we obtain
\[
d_T(\xi,A)\geq  \min_{\delta\in A} \frac{ F(\xi)-F(\delta)}{k\sqrt{\sum_{y\in\xi}F(y,\xi)^2}} \geq \min_{\delta\in A} \frac{ F(\xi) - F(\delta)}{k\sqrt{c}F(\xi)^{\alpha/2}}.
\]
Now, to prove the first inequality, let
\[
A = \{\delta\in\Nb_{\rm fin} \ : \ F(\delta)\leq\mm\}.
\]
Then, since the map $s\mapsto s / (s+\mm)^{\alpha/2}$ is increasing, we have
\[
d_T(\xi,A) \geq\frac{F(\xi) - \mm}{k\sqrt{c}F(\xi)^{\alpha/2}}\geq \frac{r}{k\sqrt{c}(r+\mm)^{\alpha/2}}
\]
for $\P_\eta$-a.e. $\xi\in\Nb_{\rm fin}$ that satisfies $F(\xi)\geq \mm + r$. This observation together with (\ref{concentration}) and the fact that $\P(\eta\in A)\geq \tfrac{1}{2}$ yields
\begin{align*}
\P(F(\eta )\geq \mm + r) &\leq \P\left(d_T(\eta,A) \geq \frac{r}{k\sqrt{c}(r+\mm)^{\alpha/2}}\right)\\
&\leq 2 \exp\left(-\frac{r^2}{4k^2c(r+\mm)^{\alpha}}\right).
\end{align*}

To prove the second inequality, let $r\geq 0$ and
\[
A = \{\delta\in\Nb_{\rm fin} \ : \ F(\delta)\leq\mm - r\}.
\]
If $\P(F(\eta)\leq \mm - r) = \P(\eta\in A) = 0$, the desired inequality holds trivially, so assume that $\P(F(\eta)\leq \mm - r)>0$ and note that this implies $r\leq \mm$. Then, since the map $s\mapsto (s-(\mm-r))/s^{\alpha/2}$ is increasing, we have
\[
d_T(\xi,A) \geq\frac{F(\xi) - (\mm-r)}{k\sqrt{c}F(\xi)^{\alpha/2}}\geq
\frac{\mm - (\mm-r)}{k\sqrt{c}\mm^{\alpha/2}} = \frac{r}{k\sqrt{c}\mm^{\alpha/2}}
\]
for $\P_\eta$-a.e. $\xi\in\Nb_{\rm fin}$ that satisfies $F(\xi)\geq \mm$. Thus, it follows from (\ref{concentration}) that
\begin{align*}
 \frac{1}{2} &\leq \P(F(\eta )\geq \mm) \leq \P\left(d_T(\eta ,A) \geq
\frac{r}{k\sqrt{c}\mm^{\alpha/2}}\right)\\ &\leq \frac{1}{\P(F(\eta )\leq \mm -r)}
\exp\left(-\frac{r^2}{4k^2c\mm^{\alpha}}\right).
\end{align*}
\end{proof}

As a final preparation for the proof of Theorem \ref{thm2}, we establish the following lemma. Recall that $\M X$ is the smallest median of a random variable $X$, as defined in \eqref{MedDef}.

\begin{lem}\label{medianLemma}
Let $X$ and $X_n, n\in\N$ be random variables such that a.s. $X_{n+1}\geq X_n$ for all $n\in\N$ and $X_n\overset{a.s.}{\to} X$. Then there exists a non-decreasing sequence $(\mm_n)_{n\in\N}$ where $\mm_n$ is a median of $X_n$ such that $\lim_{n\to\infty} \mm_n = \M X$.
\end{lem}

\begin{proof}
For a random variable $Z$, let
\begin{align*}
\hat{\M} Z = \sup\{x\in\R : \P(Z\geq x) \geq 1/2\}<\infty.
\end{align*}
Note that $(\hat{\M} X_n)_{n\in\N}$ is a non-decreasing sequence and that $\hat{\M} X_n \leq \hat{\M} X$ for all $n\in\N$, hence $(\hat{\M} X_n)_{n\in\N}$ is convergent. We claim that
\begin{align}\label{medianClaim}
\lim_{n\to\infty} \hat{\M} X_n \geq \M X.
\end{align}
To see this, let $x\in\R$ be such that $\P(X \leq x)<1/2$. Since almost sure convergence of the $X_n$ implies convergence in distribution, we have by the Portmanteau theorem (see e.g. \cite[Theorem 4.25]{Ka_2002}) that
\begin{align*}
\limsup_{n\to\infty} \P(X_n\leq x) \leq \P(X\leq x)<1/2.
\end{align*}
Hence, for sufficiently large $n$, one has $\P(X_n\leq x)<1/2$. This implies that for sufficiently large $n$, we have $\P(X_n\geq x) \geq 1/2$ and thus $\hat{\M} X_n \geq x$. From these considerations it follows that
\begin{align*}
\lim_{n\to\infty} \hat{\M} X_n &\geq \sup\{x\in\R: \P(X\leq x)<\tfrac 12\} \\ &= \inf\{x\in\R: \P(X\leq x)\geq \tfrac 12\} = \M X.
\end{align*}
Hence (\ref{medianClaim}) is established. Now, for any $n\in\N$, either $\M X_n = \hat{\M} X_n$ is the unique median of $X_n$ or all elements in the interval $[\M X_n, \hat{\M} X_n)$ are medians of $X_n$, where $\M X_n$ and $\hat{\M} X_n$ are non-decreasing in $n$. Taking (\ref{medianClaim}) into account as well as the fact that $\M X_n \leq \M X$ for all $n\in\N$, the result follows.
\end{proof}

\begin{proof}[Proof of Theorem \ref{thm2}]
For any $n\in\N$ let $\eta_n = \eta\cap B(0,n)$. Then $\eta_n$ is a Poisson point process with finite intensity measure that is given by $\mu_n(A) = \mu(A\cap B(0,n))$ for any Borel set $A\subseteq \R^d$. We define for any $n\in\N$ the random variable $F_n$ in terms of the functional $F:\Nb\to\R\cup\{\infty\}$ by $F_n = F(\eta_n)$. One easily observes that, since $F$ is a U-statistic with non-negative kernel, we have almost surely $F_{n+1}\geq F_n$ for all $n\in\N$ and also $F_n\overset{a.s.}{\to} F$. According to Lemma \ref{medianLemma} we can choose a non-decreasing sequence $\mm_n, n\in\N$ such that $\mm_n$ is a median of $F_n$ satisfying $\lim_{n\to\infty} \mm_n = \M F$. Let $r\geq 0$. By virtue of Proposition \ref{propMedian}, for any $n\in\N$ we have
\begin{align*}
\P(F_n - \mm_n \geq r) \leq 2 \exp\left(-\frac{r^2}{4k^2c(r+\mm_n)^{\alpha}}\right) \leq 2 \exp\left(-\frac{r^2}{4k^2c(r+\M F)^{\alpha}}\right).
\end{align*}
The sequence of random variables $F_n - \mm_n$ converges almost surely (and thus in distribution) to $F - \M F$. Therefore, by the Portmanteau theorem (see e.g. \linebreak \cite[Theorem 4.25]{Ka_2002}) we have
\begin{align*}
\P(F - \M F > r) \leq \liminf_{n\to\infty} \P(F_n - \mm_n > r) \leq 2 \exp\left(-\frac{r^2}{4k^2c(r+\M F)^{\alpha}}\right).
\end{align*}
The inequality for the lower tail follows analogously.
\end{proof}

\subsection*{Proof of Proposition \ref{optProp}} The proof presented below generalizes ideas from \linebreak \cite[Section 6.1]{BP_2015}. Recall that $F$ is a U-statistic of order $k$ satisfying (K1) to (K3) and that $I$ and $a>0$ are chosen such that
\begin{align*}
\tag{\ref{tailCond}}\P(F\geq M + r) \leq \exp(-I(r)),
\end{align*}
where $M$ is either the mean or a median of $F$, and
\begin{align*}
\tag{\ref{ICond}}\liminf_{r\to\infty} I(r)/r^a > 0.
\end{align*}
\begin{proof}[Proof of Proposition \ref{optProp}]
Choose some $x\in\R^d$ such that $q = \mu(B(x,\rho_F/2))>0$. Then $Z:=\eta(B(x,\rho_F/2))$ is a Poisson random variable with mean $q$. Since the diameter of $B(x,\rho_F/2)$ equals $\rho_F$, the assumptions (K1) and (K3) yield that for any $\xb \in \eta_{\neq}^k \cap B(x,\rho_F/2)$ one has $f(\xb)\geq m_F>0$. It follows that
\begin{align*}
F\geq m_Fk!\binom{Z}{k}.
\end{align*}
Hence, there exists a constant $A>0$ such that for sufficiently large $r$,
\begin{align*}
A Z^k \geq r \ \ \text{implies} \ \ F\geq r.
\end{align*}
Thus, for sufficiently large $r$,
\begin{align}\label{upperTailLowerBound}
\P(F\geq M + r) \geq \P(A Z^k \geq M + r) = \P(Z \geq \tau(r)),\end{align}
where $\tau(r) = A^{-1/k} (r+M)^{1/k}$. The tail asymptotic of a Poisson random variable is well known and one has $\P(Z\geq s) \sim \exp(-s \log(s/q)-q)$ as $s\to\infty$, see e.g. \cite{G_1987}. Here the symbol $\sim$ denotes the asymptotic equivalence relation. Combining this with (\ref{upperTailLowerBound}), we obtain
\begin{align*}
\liminf_{r\to\infty} \frac{\P(F\geq M + r)}{\exp(-\tau(r)\log(\tau(r)/q) - q)} \geq 1.
\end{align*}
Hence, taking (\ref{tailCond}) into account, there exists a constant $B>0$ such that for sufficiently large $r$,
\begin{align*}
\tau(r)\log(\tau(r)/q) + q \geq I(r) -B.
\end{align*}
By virtue of (\ref{ICond}), dividing this inequality by $r^a$ and taking the limit yields
\begin{align} \label{positiveLiminf}
\liminf_{r\to\infty} \frac{\tau(r)\log(\tau(r)/q)}{r^a} > 0.
\end{align}
Moreover, writing $C = A^{-1/k}>0$, we have
\begin{align*}
\frac{\tau(r)\log(\tau(r)/q)}{r^a} \sim C r^{1/k-a} \log(C(r+M)^{1/k} /q).
\end{align*}
The result follows since $a>1/k$ would imply that the RHS in the above display converges to $0$, contradicting (\ref{positiveLiminf}).
\end{proof}

\subsection*{Proof of Proposition \ref{IntegrableProp} and Theorem \ref{SGCExp}}
The upcoming reasoning is partially inspired by the proof of \cite[Proposition 3.1]{P_2003}. First of all, we need to introduce the quantities $\lVert f_n \rVert_n^2$ that are crucial ingredients for the stochastic analysis of Poisson processes using Malliavin Calculus, see e.g. \cite{L_2014} for details. According to \linebreak \cite[Lemma 3.5]{RS_2013}, for a square-integrable U-statistic with kernel $f$ these quantities can be explicitly written as
\begin{align} \label{normFn}
\lVert f_n\rVert _n^ 2 = \binom{k}{n}^2 \int_{(\R^d)^n} \left( \int_{(\R^d)^{k-n}} f(\yb_n,\xb_{k-n})d\mu^{k-n}(\xb_{k-n}) \right)^2 d\mu^n(\yb_n),
\end{align}
where we use the abbreviations $\yb_n = (y_1,\ldots,y_n)$ and $\xb_{k-n} = (x_{n+1},\ldots,x_k)$. Also according to \cite[Lemma 3.5]{RS_2013}, one has that the variance of a square-integrable \linebreak U-statistic $F$ is now given by
\begin{align}\label{VarFormula}
\V F = \sum_{n=1}^k n! \lVert f_n \rVert_n^2.
\end{align}
This formula will be crucial for analysing the variance asymptotics of subgraph counts. For the remainder of the section, let the assumptions and notations of Section \ref{ABSC} prevail. In particular, recall that $m$ denotes the Lebesgue density of the measure $\mu$ and note that we will use the notation
\begin{align*}
m^{\otimes l}(\textbf{z}) = m^{\otimes l} (z_1,\ldots,z_l) = \prod_{i=1}^l m(z_i)
\end{align*}
for any $l\in\N$ and $\textbf{z} = (z_1,\ldots,z_l) \in (\R^d)^l$. The following two lemmas are used in the proofs of Proposition \ref{IntegrableProp} and Theorem \ref{SGCExp}.

\begin{lem} \label{formulasExpNormFn}
Assume that the subgraph count $N_t$ is square-integrable. Then the corresponding quantities $\lVert f_n \rVert_n^2$ given by formula (\ref{normFn}) can be written as
\begin{align} \label{normFnFormula}
c_n(t) \int_{\R^d\times B(0,\Theta)^{n-1}} I_t^y(\yb_n) \left(\int_{B(0,\Theta)^{k-n}} I^x_t( \xb_{k-n}) \ J(\yb_n,  \xb_{k-n}) \ d  \xb_{k-n} \right)^2 d\yb_n,
\end{align}
where
\begin{align*}
c_n(t)& = \frac{t^{2k-n}\rho_t^{d(2k-n-1)}}{(k!)^2} \binom{k}{n}^2,\\
\Theta &= \diam(H)\theta,\\
I^y_t(\yb_n) &= m^{\otimes n}(y_1,\rho_ty_2 + y_1,\ldots, \rho_t y_n + y_1),\\
I^x_t( \xb_{k-n}) &= m^{\otimes k-n}(\rho_tx_{n+1} + y_1,\ldots, \rho_tx_{k} + y_1),\\
J(\yb_n, \xb_{k-n}) &= |\{\text{subgraphs} \ H' \ \text{of} \ G_S(0,y_2 ,\ldots, y_n,x_{n+1}, \ldots, x_{k}): H'\cong H\}|.
\end{align*}
\end{lem}

\begin{proof}
Recall that the kernel of the subgraph count $N_t$ is given by
\begin{align*}
f_t( \zb) = \frac{|\{\text{subgraphs} \  H' \ \text{of} \ G_{\rho_t S}( \zb): H'\cong H\}|}{k!}, \ \  \zb\in(\R^d)^k.
\end{align*}
Since $N_t$ is assumed to be square-integrable, formula (\ref{normFn}) holds. The integral on the RHS of (\ref{normFn}) can be written as
\begin{align*}
t^{2k-n} \int_{(\R^d)^n} m^{\otimes n}(\yb_n)\left(\int_{(\R^d)^{k-n}} m^{\otimes k-n}( \xb_{k-n}) f_t(\yb_n, \xb_{k-n}) \ d \xb_{k-n}\right)^2 \ d\yb_{n}.
\end{align*}
The change of variables
\[
 \xb_{k-n} \mapsto (\rho_t x_{n+1} + y_1, \ldots, \rho_t x_{k} + y_1)
\]
now gives that the above expression equals
\begin{align*}
\frac{t^{2k-n}\rho_t^{2d(k-n)}}{(k!)^2} \int_{(\R^d)^n} m^{\otimes n}(\yb_n) \left(\int_{(\R^d)^{k-n}} I^x_t( \xb_{k-n}) \ J^x_t(\yb_n,  \xb_{k-n}) \ d  \xb_{k-n} \right)^2 d\yb_n,
\end{align*}
where
\begin{align*}
&J^x_t(\yb_n,  \xb_{k-n}) = |\{\text{subgraphs } H' \text{ of } G_{\rho_tS}(\yb_n, \rho_t x_{n+1} + y_1,\ldots, \rho_t x_{k} + y_1) : H'\cong H\}|.
\end{align*}
We perform a further change of variables for the outer integral, namely
\begin{align*}
\yb_{n} \mapsto (y_1,\rho_t y_2 + y_1, \ldots, \rho_t y_n + y_1).
\end{align*}
This yields that the quantity in question equals
\begin{align*}
\frac{t^{2k-n}\rho_t^{d(2k-n-1)}}{(k!)^2} \int_{(\R^d)^n} I_t^y(\yb_n) \left(\int_{(\R^d)^{k-n}} I^x_t( \xb_{k-n}) \ J(\yb_n,  \xb_{k-n}) \ d  \xb_{k-n} \right)^2 d\yb_n,
\end{align*}
where one should notice that
\begin{align*}
&J^x_t(y_1,\rho_t y_2 + y_1,\ldots, \rho_t y_n + y_1, \xb_{k-n})\\
& =|\{\text{subgraphs} \ H' \ \text{of} \ G_S(0,y_2 ,\ldots, y_n,x_{n+1}, \ldots, x_{k}): H'\cong H\}|\\
& = J(\yb_n,\xb_{k-n}).
\end{align*}
Also, we see that $J(\yb_n, \xb_{k-n}) = 0$ whenever $\lVert x_i\rVert > \diam(H)\theta = \Theta$ for some \linebreak $n+1\leq i\leq k$. Hence, the inner integral is actually an integral over $B(0,\Theta)^{k-n}$. Similarly, the outer integral is an integral over $\R^d \times B(0,\Theta)^{n-1}$.
\end{proof}
In the next statement we use the abbreviation $\xb = (x_1,\ldots,x_k)$. Note that here we do \emph{not} assume that the considered subgraph count $N_t$ is necessarily square-integrable.

\begin{lem} \label{ExpFormula}
The expectation of $N_t$ can be written as
\begin{align} \label{ExpFormulaEq}
\E N_t = \frac{t^k\rho_t^{d(k-1)}}{k!} \int_{\R^d\times B(0,\Theta)^{k-1}} I_t( \xb) \ J( \xb) \ d  \xb,
\end{align}
where
\begin{align}\label{ItDef}
I_t( \xb) &= m^{\otimes k}(x_1,\rho_t x_2 + x_1, \ldots, \rho_t x_{k} + x_1),\\
\nonumber J( \xb) &= |\{\text{subgraphs} \ H' \ \text{of} \ G_{S}(0,x_2,\ldots,x_k): H'\cong H\}|.
\end{align}
In (\ref{ExpFormulaEq}), the LHS is finite if and only if the RHS is finite.
\end{lem}

\begin{proof}
Using the Slivniak-Mecke formula (see e.g. \cite[Corollary 3.2.3]{SW_2008}) we can write
\[
\E {N_t} = \frac{t^k}{k!} \int_{(\R^d)^k} m^{\otimes k}( \xb) \ |\{\text{subgraphs} \  H' \ \text{of} \ G_{\rho_t S}( \xb): H'\cong H\}| \ d \xb.
\]
Now, the change of variables
\[
 \xb = (x_1,\ldots,x_k)\mapsto (x_1,\rho_t x_2 + x_1, \ldots, \rho_t x_{k} + x_1)
\]
followed by a reasoning very similar to the one in the proof of Lemma \ref{formulasExpNormFn} yields the result.
\end{proof}

\begin{proof}[Proof of Proposition \ref{IntegrableProp}]
Let the notation of Lemma \ref{ExpFormula} prevail. The assumption (\ref{desityCond}) implies that for any $t\in\N$ and $\xb\in\R^d\times B(0,\Theta)^{k-1}$,
\begin{align*}
c^{-k}m(x_1)^k \leq I_t(\xb) \leq c^km(x_1)^k.
\end{align*}
Using Lemma \ref{ExpFormula}, one obtains
\begin{align*}
c^{-k} A \int_{\R^d}m(x)^k\ dx  \leq \frac{k!}{t^k\rho_t^{d(k-1)}}\E N_t \leq c^k A \int_{\R^d} m(x)^k\ dx,
\end{align*}
where 
\begin{align*}
A = \int_{B(0,\Theta)^{k-1}} |\{\text{subgraphs} \ H' \ \text{of} \ G_S(0,x_2 , \ldots, x_{k}): H'\cong H\}|\ d  \xb > 0.
\end{align*}
Hence, it follows that $\int m(x)^k dx<\infty$ is equivalent to integrability of $N_t$. According to Theorem \ref{main}, the subgraph count $N_t$ is integrable if and only if it is almost surely finite. The result follows.
\end{proof}

\begin{proof}[\it{Proof of Theorem \ref{SGCExp}}]
[Proof of (i)] 
Let the notation of Lemma \ref{ExpFormula} prevail. Since by assumption the density $m$ is continuous, we have for $I_t$ given by (\ref{ItDef}) and for any $\xb\in(\R^d)^k$,
\[
\lim_{t\to\infty}I_t( \xb) = m(x_1)^k.
\]
Now, the condition (\ref{desityCond}) guarantees that for any $\xb\in\R^d\times B(0,\Theta)^{k-1}$,
\begin{align*}
I_t(\xb) \leq c^km(x_1)^k.
\end{align*}
Moreover, by virtue of assumption (\ref{integrableCond}) and boundedness of $J$,
\begin{align*}
K := \int_{\R^d\times B(0,\Theta)^{k-1}} m(x_1)^k J(\xb) \ d\xb < \infty.
\end{align*}
Hence, Lemma \ref{ExpFormula} together with the dominated convergence theorem yield that the sequence $\E {N_t} / (t^k \rho_t^{d(k-1)})$ converges to $K/k!>0$.
\medskip

[Proof of (ii)] Since $\M N_t$ is a median of $N_t$, in the case $\M N_t > \E N_t$ one has
\begin{align*}
\frac 12 &\leq \P(N_t \geq \M N_t) = \P(N_t - \E N_t \geq \M N_t- \E N_t)\\ &\leq \P(|N_t - \E N_t| \geq |\M N_t- \E N_t|) \leq \frac{\V N_t}{(\M N_t - \E N_t)^2},
\end{align*}
where for the last inequality we used Chebyshev. Similarly, if $\M N_t < \E N_t$,
\begin{align*}
\frac 12 &\leq \P(N_t \leq \M N_t) \leq \P(|N_t - \E N_t| \geq |\M N_t- \E N_t|) \leq \frac{\V N_t}{(\M N_t - \E N_t)^2}.
\end{align*}
We see that in any case one has
\begin{align*}
\left| \frac{\M N_t}{\E N_t} - 1 \right| \leq \frac{\sqrt{2\V N_t}}{\E N_t}.
\end{align*}
Now, the statements (i) and (iii) of the theorem yield
\begin{align}\label{asympVbyE}
\frac{\sqrt{2\V N_t}}{\E N_t} \sim \sqrt{\frac{2}{a^2}\sum_{n=1}^k \frac{(t\rho_t^d)^{k-n}}{t^k\rho_t^{d(k-1)}} A^{(n)}} = \sqrt{\frac{2}{a^2}\sum_{n=1}^k \rho_t^{d(1-n/k)} (t^k\rho_t^{d(k-1)})^{-n/k} A^{(n)}}.
\end{align}
The assumption $\lim_{t\to\infty}\E N_t = \infty$ together with statement (i) of the theorem and $\lim_{t\to\infty} \rho_t = 0$ implies that this tends to $0$ as $t\to\infty$.
\medskip

[Proof of (iii)] First note that by Proposition \ref{IntegrableProp} together with Theorem \ref{main} the assumption (\ref{integrableCond}) guarantees that all moments of the subgraph counts $N_t$ exist. In particular, the $N_t$ are square-integrable, thus formula (\ref{VarFormula}) is in order. So, to analyse the asymptotic behaviour of $\V N_t$, we need to analyse the asymptotics of the quantities $\lVert f_n \rVert_n^2$. This will be done by proving convergence of the integrals that appear in (\ref{normFnFormula}),
\begin{align*}
\int_{\R^d\times B(0,\Theta)^{n-1}} I_t^y(\yb_n) \left(\int_{B(0,\Theta)^{k-n}} I^x_t( \xb_{k-n}) \ J(\yb_n,  \xb_{k-n}) \ d  \xb_{k-n} \right)^2 d\yb_n.
\end{align*}
We first remark that by continuity of the density $m$, for any $ \xb_{k-n}\in(\R^d)^{k-n}$,
\[
\lim_{t\to\infty}I^x_t(\xb_{k-n}) = m(y_1)^{k-n}.
\]
Since $J$ and $m$ are bounded, the dominated convergence theorem applies and gives
\begin{align*}
&\lim_{t\to\infty} \int_{B(0,\Theta)^{k-n}} I^x_t( \xb_{k-n}) \ J(\yb_n,  \xb_{k-n}) \ d  \xb_{k-n}\\ &= \int_{B(0,\Theta)^{k-n}} m(y_1)^{k-n} \ J(\yb_n,  \xb_{k-n}) \ d  \xb_{k-n}.
\end{align*}
Now, it follows from the assumption (\ref{desityCond}) that for any $ \xb_{k-n}\in B(0,\Theta)^{k-n}$ and $\yb_{n}\in \R^d\times B(0,\Theta)^{n-1}$,
\begin{align*}
I^x_t( \xb_{k-n}) \leq c^{k-n} m(y_1)^{k-n} \ \ \text{and} \ \ I^y_t(\yb_{n}) \leq c^{n} m(y_1)^{n}.
\end{align*}
Moreover, we have that
\begin{align*}
&\int_{\R^d\times B(0,\Theta)^{n-1}}m(y_1)^{n} \left(\int_{B(0,\Theta)^{k-n}} m(y_1)^{k-n} \ J(\yb_n,  \xb_{k-n}) \ d  \xb_{k-n} \right)^2 \ d\yb_n\\
&\leq (\Theta^{d}\kappa_d)^{2k-n-1} \lVert J\rVert_\infty^2 \int_{\R^d}m(x)^{2k-n}  \ d x < \infty,
\end{align*}
where $\kappa_d$ denotes the Lebesgue measure of the unit ball in $\R^d$. Here the finiteness of the latter integral follows from assumption (\ref{integrableCond}) since $m$ is bounded and $2k-n\geq k$. Furthermore, by continuity of $m$, one has for any $\yb_n\in \R^d\times B(0,\Theta)^{n-1}$,
\begin{align*}
\lim_{t\to\infty} I^y_t(\yb_n) = m(y_1)^n.
\end{align*}
We conclude that the dominated convergence theorem applies and gives
\begin{align*}
&\lim_{t\to\infty} \int_{\R^d\times B(0,\Theta)^{n-1}} I_t^y(\yb_n) \left(\int_{B(0,\Theta)^{k-n}} I^x_t( \xb_{k-n}) \ J(\yb_n,  \xb_{k-n}) \ d  \xb_{k-n} \right)^2 d\yb_n\\
&= \int_{\R^d} m(x)^{2k-n} dx \int_{B(0,\Theta)^{n-1}} \left( \int_{B(0,\Theta)^{k-n}}  J^{0}(\yb_{n-1}, \xb_{k-n}) \ d \xb_{k-n} \right)^2 d \yb_{n-1}\\ &=: K^{(n)},
\end{align*}
where
\begin{align*}
J^{0}(\yb_{n-1}, \xb_{k-n}) = |\{\text{subgraphs} \ H' \ \text{of} \ G_S(0,\yb_{n-1}, \xb_{k-n}): H'\cong H\}|.
\end{align*}
Hence, by virtue of (\ref{VarFormula}) and (\ref{normFnFormula}) the variance can be written as
\begin{align*}
\V N_t &= \sum_{n=1}^k t^{2k-n}\rho_t^{d(2k-n-1)}\frac{n!}{(k!)^2}\binom{k}{n}^2 K^{(n)}_t,\\
&= t^k\rho_t^{d(k-1)}\sum_{n=1}^k (t\rho_t^d)^{k-n}\frac{n!}{(k!)^2} \binom{k}{n}^2 K^{(n)}_t.
\end{align*}
where $K^{(n)}_t$ is such that $\lim_{t\to\infty} K^{(n)}_t = K^{(n)}\in(0,\infty)$. The result follows.
\end{proof}

\subsection*{Proof of Theorem \ref{SLLN}}
In the upcoming proof of the strong law for subgraph counts we use the following consequence of the Borel-Cantelli lemma (see e.g. \linebreak \cite[Theorem 3.18]{Ka_2002}), well known from basic probability theory.

\begin{lem} \label{aslem}
Let $(X_n)_{n\in\N}$ be a sequence of real random variables and let $(a_n)_{n\in\N}$ be a sequence of real numbers converging to some $a\in\R$. Assume that for any $\epsilon>0$,
\begin{align*}
\sum_{n=1}^\infty \P(|X_n - a_n|\geq\epsilon) < \infty.
\end{align*}
Then
\begin{align*}
X_n \overset{a.s.}{\longrightarrow} a \ \ \text{as} \ \ n\to\infty.
\end{align*}
\end{lem}

\begin{proof}[Proof of Theorem \ref{SLLN}]
It follows from Theorem \ref{main} that for any real number $\epsilon>0$,
\begin{align*}
\P\left(\left|\frac{N_t}{t^k\rho_t^{d(k-1)}} - \frac{\E N_t}{t^k\rho_t^{d(k-1)}} \right|\geq \epsilon\right) \leq  p_u(t,\epsilon) + p_l(t,\epsilon),
\end{align*}
where
\begin{align*}
p_u(t,\epsilon) &:= \exp\left(-t\rho_t^{d(k-1)/k}\frac{\left(\left(\tfrac{\E N_t}{t^k\rho_t^{d(k-1)}} + \epsilon\right)^{\tfrac{1}{2k}}-\left(\tfrac{\E N_t}{t^k\rho_t^{d(k-1)}}\right)^{\tfrac{1}{2k}}\right)^2}{2k^2c_d}\right),\\
p_l(t,\epsilon) &:= \exp\left(-\frac{t^{2k}\rho_t^{2d(k-1)}\epsilon^2}{2k \V N_t}\right).
\end{align*}
Now, by Theorem \ref{SGCExp} (i) the sequence $\E N_t / (t^k \rho_t^{d(k-1)})$ converges to a positive real number. This implies that we can choose a constant $C>0$ such that for sufficiently large $t$,
\begin{align*}
p_u(t,\epsilon) \leq \exp\left(-t\rho_t^{d(k-1)/k} C\right).
\end{align*}
Moreover, assumption \eqref{regime} guarantees the existence of a constant $C'>0$ such that for sufficiently large $t$,
\begin{align*}
\exp\left(-t\rho_t^{d(k-1)/k} C\right) \leq \exp\left(-t^{\gamma/k}C'\right).
\end{align*}
We conclude that $\sum_{t\in\N} p_u(t,\epsilon) < \infty$. Furthermore, similarly as in \eqref{asympVbyE}, it follows from Theorem \ref{SGCExp} (iii) that
\begin{align*}
\frac{\V N_t}{t^{2k}\rho_t^{2d(k-1)}} \sim \sum_{n=1}^k \rho_t^{d(1-n/k)} (t^k\rho_t^{d(k-1)})^{-n/k} A^{(n)}.
\end{align*}
Invoking assumption \eqref{regime} yields existence of a constant $C''>0$ such that for sufficiently large $t$,
\begin{align*}
\sum_{n=1}^k \rho_t^{d(1-n/k)} (t^k\rho_t^{d(k-1)})^{-n/k} A^{(n)} \leq \sum_{n=1}^k \rho_t^{d(1-n/k)} (C'' t^\gamma)^{-n/k} A^{(n)}.
\end{align*}
Now, the sequence $\rho_t$ tends to $0$ and thus it is bounded above. Moreover, for any $t\in\N$ one has $t^{-k\gamma/k}\leq t^{-(k-1)\gamma/k}\leq \ldots \leq t^{-\gamma /k}$. From this together with the last two displays we derive existence of a constant $C'''>0$ such that for large $t$,
\begin{align*}
\frac{\V N_t}{t^{2k}\rho_t^{2d(k-1)}} \leq C''' t^{-\gamma/k}.
\end{align*}
Hence, for sufficiently large $t$,
\begin{align*}
p_l(t,\epsilon) \leq \exp\left(-\frac{t^{\gamma/k}\epsilon^2}{2kC'''}\right).
\end{align*}
It follows that $\sum_{t\in\N} p_l(t,\epsilon)<\infty$. Applying Lemma \ref{aslem} yields the result.
\end{proof}

\bigskip

\renewcommand{\bibname}{References}
\defbibheading{bibliography}[\refname]{\section*{#1}}
\printbibliography
\end{document}